\documentclass[a4paper]{amsart}
\usepackage{amsfonts,amsmath,amssymb,amscd,amstext,amsbsy,latexsym}
\newtheorem{thm}{Theorem}[section]
\newtheorem{cor}[thm]{Corollary}
\newtheorem{lem}[thm]{Lemma}
\newtheorem{prop}[thm]{Proposition}
\theoremstyle{definition}

\theoremstyle{remark}
\newtheorem{rem}[thm]{Remark}
\newtheorem{exm}[thm]{Example}
\numberwithin{equation}{section}

\newcommand{\Z}{{\bf Z}}

\begin{document}

\title{Conch Maximal Subrings}%
\author{Alborz Azarang}
\keywords{Maximal subring, Conch subring, Integral element, Conductor ideal}%
\subjclass[2010]{13B02, 13B21, 13B30, 13A05, 13A15, 13A18}%
\date{5 Jun 2020}
\maketitle

\centerline{Department of Mathematics, Faculty of Mathematical Sciences and Computer,}
\centerline{Shahid Chamran Universityof Ahvaz, Ahvaz-Iran}
\centerline{a${}_{-}$azarang@scu.ac.ir}
\centerline{ORCID ID: orcid.org/0000-0001-9598-2411}

\begin{abstract}
It is shown that if $R$ is a ring, $p$ a prime element of an integral domain $D\leq R$ with $\bigcap_{n=1}^\infty p^nD=0$ and $p\in U(R)$, then $R$ has a conch maximal subring (see \cite{faith}). We prove that either a ring $R$ has a conch maximal subring or $U(S)=S\cap U(R)$ for each subring $S$ of $R$ (i.e., each subring of $R$ is closed with respect to taking inverse, see \cite{invsub}). In particular, either $R$ has a conch maximal subring or $U(R)$ is integral over the prime subring of $R$. We observe that if $R$ is an integral domain with $|R|=2^{2^{\aleph_0}}$, then either $R$ has a maximal subring or $|Max(R)|=2^{\aleph_0}$, and in particular if in addition $dim(R)=1$, then $R$ has a maximal subring. If $R\subseteq T$ be an integral ring extension, $Q\in Spec(T)$, $P:=Q\cap R$, then we prove that whenever $R$ has a conch maximal subring $S$ with $(S:R)=P$, then $T$ has a conch maximal subring $V$ such that $(V:T)=Q$ and $V\cap R=S$. It is shown that if $K$ is an algebraically closed field which is not algebraic over its prime subring and $R$ is affine ring over $K$, then for each prime ideal $P$ of $R$ with $ht(P)\geq dim(R)-1$, there exists a maximal subring $S$ of $R$ with $(S:R)=P$. If $R$ is a normal affine integral domain over a field $K$, then we prove that $R$ is an integrally closed maximal subring of a ring $T$ if and only if $dim(R)=1$ and in particular in this case $(R:T)=0$.
\end{abstract}

\section{Introduction}
Following Faith \cite{faith}, a subring $V$ of a commutative ring $T$ is called a conch subring if there exists a unit $x\in T$ such that $x^{-1}\in V$ but $x\notin V$ and $V$ is maximal respect to this property (i.e., excluding $x$ and including $x^{-1}$), $V$ is called $x$-conch subring or it is said that $V$ conches $x$ in $T$. In other words, $V$ is a $x$-conch subring of a ring $T$ if and only if $V$ is maximal in the set $\{R\ | R\ \text{is a subring of}\ T,\ {\bf Z}[x^{-1}]\subseteq R,\ x\notin R\}$, where $\Z$ denotes the prime subring of $R$. But note that clearly, the previous set is nonempty if and only if $x\notin Z[x^{-1}]$ which is equivalent to the fact that $x$ is not integral over $\Z$. Krull proved that an integral domain $R$ is integrally closed if and only if $R$ is the intersection of the conch subrings of $K=Q(R)$ that contains $R$, and moreover, each conch subring $V$ of a field $K$ is a chain ring and therefore is a valuation domain (i.e., for each $x\in K$, either $x\in V$ or $x^{-1}\in K$), see \cite[P.110]{krul}. In \cite{faith}, Faith studied conch subrings for general commutative rings by the use of Manis valuations rings (or maximal pairs). He proved that if a subring $A$ conches $x$ in a ring $Q$, then $A$ is integrally closed and $(A,\sqrt{x^{-1}A})$ is a maximal pair of $Q$; in particular, $A$ is a valuation ring. Conversely, a maximal pair $(A,P)$ comes from a conch subring of $Q$ if and only if $P=\sqrt{x^{-1}A}$ for some unit $x\in Q$ ($x^{-1}\in A$) [Conch Ring Theorem]. Moreover, if $R$ is a subring of a ring $Q$, then the intersection $Conch_Q(R)$ of the conch subrings of $Q$ containing $R$ is integrally closed, and every element of $Conch_Q^*(R):=Conch_Q(R)\cap U(Q)$ is integral over $R$. More interestingly, in \cite{grifn}, Griffin proved that if $Q$ is VNR or has few zero-divisor (i.e., the set of all zero divisor of $Q$, $Zd(Q)$, is a finite union of prime ideals of $Q$) and $R$ is integrally closed in $Q=Q(R)$, then $R$ is the intersection of valuation subrings of $Q$. In \cite[Section 10]{faith}, Faith studied conch subrings which are maximal subrings, i.e., if $A$ conches $x$ in $Q$, then $A$ is a maximal subring of $Q$ if and only if $Q=A[x]$.\\

In this paper, motivated by the above facts and the existence of maximal subring, we are interesting to prove that if $R$ is a ring then either $R$ has a conch maximal subring or each subring of $R$ is closed with respect to taking inverse, i.e., $U(S)=S\cap U(R)$ for each subring $S$ of $R$, see \cite{invsub}. In particular, for each ring $R$ either $U(R)$ is integral over $\Z$, or $R$ has a conch maximal subring.\\

All rings in this note are commutative with $1\neq 0$. All subrings, ring extensions, homomorphisms and modules are unital. A proper subring $S$ of a ring $R$ is called a maximal subring if $S$ is maximal with respect to inclusion in the set of all proper subrings of $R$. Not every ring possesses maximal subrings (for example the algebraic closure of a finite field has no maximal subrings, see \cite[Remark 1.13]{azkrm}; also see \cite[Example 2.6]{azkrm2} and \cite[Example 3.19]{azkarm4} for more examples of rings which have no maximal subrings). A ring which possesses a maximal subring is said to be submaximal, see \cite{azarang3}, \cite{azkrm}  and \cite{azkarm4}. If $S$ is a maximal subring of a ring $R$, then the extension $S\subseteq R$ is called a minimal ring extension (see \cite{frd}) or an adjacent extension too (see \cite{adjex}). Minimal ring extensions first appears in \cite{gilmer}, for studying integral domains with a unique minimal overring. Next in \cite{frd}, these extensions are fully studied and some essential facts are obtained. The following result, whose proof could be found in \cite{frd} is needed. Before presenting it, let us recall that whenever $S$ is a maximal subring of a ring $R$, then one can easily see that either $R$ is integral over $S$ or $S$ is integrally closed in $R$. If $S\subseteq R$ is a ring extension, then the ideal $(S:R)=\{x\in R\ |\ Rx\subseteq S\}$ is called the conductor of the extension $S\subseteq R$.

\begin{thm}\label{pret1}
Let $S$ be a maximal subring of a ring $R$. Then the following statements hold.
\begin{enumerate}
\item $(S:R)\in Spec(S)$.
\item $(S:R)\in Max(S)$ if and only if $R$ is integral over $S$.
\item If $S$ is integrally closed in $R$, then $(S:R)\in Spec(R)$. Moreover if $x,y\in R$ and $xy\in S$, then either $x\in S$ or $y\in S$.
\item There exist a unique maximal ideal $M$ of $S$ such that $S_M$ is a maximal subring of $R_M$ and for each $P\in Spec(S)\setminus\{M\}$, $S_P=R_P$ ($M$ is called the crucial maximal ideal of the extension $S\leq R$ and $(S:R)\subseteq M$. Moreover, $S\leq R$ is integral (resp., integrally closed) extension if and only if $S_M\leq R_M$ is integral (resp., integrally closed) extension).
\item If $S$ is an integrally closed maximal subring of $R$, $P=(S:R)$ and $M$ be the crucial maximal ideal of the extension, then $(\frac{R}{P})_{\frac{M}{P}}$ is a one dimensional valuation domain.
\end{enumerate}
\end{thm}

Note that by $(u,u^{-1})$-Lemma, one can easily see that if $S\subseteq R$ is an integral ring extension then $U(S)=S\cap U(R)$. In particular, if $S$ is a maximal subring of $R$ and $S\subseteq R$ is integral, then $U(S)=S\cap U(R)$. On the other hand, if $S$ is a maximal subring of $R$, then by the second part of $(3)$ of Theorem \ref{pret1} (or by the
use of $(u,u^{-1})$-Lemma) for each $x\in U(R)$, we infer that either $x\in S$ or $x^{-1}\in S$. Therefore, if $S$ is a maximal subring of $R$ such that there exists a unit $x$ in $R$ such that $x\in S$ but $x^{-1}\notin S$, we immediately conclude that $S$ is integrally closed in $R$. Now let us prove the following fact for conch subrings by the use of minimal ring extension.

\begin{cor}\label{corconch}
Let $T$ be a ring and $x\in U(T)$. If $R$ is a $x$-conch subring of $T$, then $x^{-1}$ is contained in exactly one prime (maximal) ideal of $R$, in other words $\sqrt{x^{-1}R}\in Max(R)$.
\end{cor}
\begin{proof}
Since $R$ is a $x$-conch subring of $T$, we immediately conclude that $R\subseteq S:=R[x]$ is a minimal ring extension. Thus by $(4)$ of Theorem \ref{pret1}, let $M$ be the crucial maximal ideal of the minimal ring extension $R\subseteq S$. We prove that $M$ is the only prime ideal of $R$ which contains $x^{-1}$. First note that since $R_M\subseteq S_M$ is a minimal ring extension ($R_M\neq S_M$), we deduce that $x\notin R_M$ and therefore $x^{-1}\in M$. Now for each prime ideal $P$ of $R$ with $P\neq M$, we have $R_P=S_P$. Therefore
$x^{-1}$ is a unit in $R_P$, hence $x^{-1}\notin P$ and we are done.
\end{proof}

Now, let us sketch a brief outline of this paper. Section 2 is devoted to the existence of integrally closed/conch maximal subrings. We prove that if $R$ is a ring, $D$ is an integral domain which is a subring of $R$, $p$ is a prime element of $D$ with $\bigcap_{i=1}^\infty p^nD=0$ and $p\in U(R)$, then $R$ has an integrally closed/conch maximal subring. In particular, if $R$ is a ring which contains an atomic (or a completely integrally closed) domain $D$ and a prime element of $D$ is invertible in $R$, then $R$ has an integrally closed/conch maximal subring. Next, we show that if $T$ is a ring, then either $T$ has an integrally closed/conch maximal subring or $U(R)=R\cap U(T)$, for each subring $R$ of $T$, which is one of the main result in Section 2. This result has many consequences. In particular, either a ring $T$ has an integrally closed/conch maximal subring or $U(T)$ is integral over $\Z$. It is observed that if $T$ is an integral domain which satisfies ACCP (resp. is a BFD, see \cite{andzaf1}), then either $T$ has maximal subring or each subring of $T$ satisfies ACCP (resp. is a BFD). We show that if $D$ is a GCD-domain which is a subring of a ring $T$, then either $T$ has a maximal subring or $U(T)\cap K=U(D)$, where $K$ is the quotient field of $D$. We prove that a ring $T$ has a conch maximal subring if and only if $\Z[U(T)]$ has an integrally closed proper subring. If $T$ is a ring without maximal subring and $Char(T)>0$, then each element of the group $U(T)$ has finite order. In particular, if in addition $U(T)$ is finitely generated, then $U(T)$ is finite. We observe that if $T$ is a residually finite ring with $Char(T)=0$ and $U(T)$ is finitely generated, then either $T$ has a maximal subring or $T=\Z$. We prove that if $T$ is a reduced ring with $|T|=2^{2^{\aleph_0}}$, then either $T$ has a maximal subring or $|Max(T)|=2^{\aleph_0}$. In particular, if $T$ is an integral domain with $|T|=2^{2^{\aleph_0}}$ and $dim(T)=1$, then $T$ has a maximal subring. It is observe that if $R=K+Rx$, where $K$ is a field which is contained in a ring $R$, and $x\notin U(R)\cup Zd(R)$, then $R$ has a maximal subring. Finally in Section 2, we prove that if $R$ is an atomic integral domain with $|R|=2^{2^{\aleph_0}}$ and $M$ is a principal maximal ideal of $R$, then $|R/M|=|R|$ and in particular $R$ has a maximal subring. Section 3, is devoted to the existence of maximal subring with certain conductor. We first prove that if $K$ is an algebraically closed field which is not algebraic over its prime subring, then for each prime ideal $P$ of $K[X]$, there exists an integrally closed (conch) maximal subring $R$ of $K[X]$ with $( R:K[X])=P$. Next, we studied the conductor of integrally closed maximal subrings for integral extensions. In fact, we show that if $R\subseteq T$ is an integral extension, $Q\in Spec(T)$, $P:=Q\cap R$, and there exists an integrally closed maximal subring $S$ of $R$ with $U(R/P)\nsubseteq S/P$ $(S:R)=P$, then $T$ has an integrally closed maximal subring $V$ with $(V:T)=Q$. Conversely, if $R\subseteq T$ is an integral extension and each maximal ideal of $T$ is conductor of an integrally closed maximal subring of $T$, then each maximal ideal of $R$ is conductor of an integrally closed maximal subring of $R$. We show that if $K$ is algebraically closed field which is not absolutely algebraic the each prime ideal $Q$ of $K[X_1,\ldots, X_n]$ with $ht(Q)\geq n-1$ is a conductor ideal of an integrally closed maximal subring of $K[X_1,\ldots, X_n]$. But for minimal ring extension of $K[X_1,\ldots, X_n]$, $n\geq 2$, a prime ideal of $K[X_1,\ldots, X_n]$, is conductor of a minimal ring extension of $K[X_1,\ldots, X_n]$ if and only if $P$ is maximal; and for $n=1$, each prime ideal of $K[X_1]$ is a conductor of a minimal ring extension of $K[X_1]$. We prove some results for the zeroness of the conductor of the minimal ring extension of the form $R\subseteq R[\frac{1}{u}]$, where $R$ is a certain ring or $u$ is a certain element of $R$.\\

Finally, let us recall some notation and definitions. As usual, let $Char(R)$, $U(R)$, $Zd(R)$, $N(R)$, $J(R)$, $Max(R)$, $Spec(R)$ and $Min(R)$, denote the characteristic, the set of all units, the set of all zero-divisors, the nil radical ideal, the Jacobson radical ideal, the set of all maximal ideals, the set of all prime ideals and the set of all minimal prime ideals of a ring $R$, respectively. We also call a ring $R$, not necessarily noetherian, semilocal (resp. local) if $Max(R)$ is finite (resp. $|Max(R)|=1$). For any ring $R$, let $\Z=\mathbb{Z}\cdot 1_R=\{n\cdot 1_R\ |\ n\in \mathbb{Z} \}$, be the prime subring of $R$. We denote the finite field with $p^n$ elements, where $p$ is prime and $n\in\mathbb{N}$, by $\mathbb{F}_{p^n}$. Fields which are algebraic over $F_p$ for some prime number $p$, are called absolutely algebraic field. If $R$ is a ring and $a\in R\setminus N(R)$, then $R_a$ denotes the ring of quotient of $R$ respect to the multiplicatively closed set $X=\{1,a,a^2,\ldots, a^n,\ldots\}$. If $D$ is an integral domain, then we denote the set of all non-associate irreducible elements of $D$ by $Irr(D)$.  Also, we denote the set of all natural prime numbers by $\mathbb{P}$. Suppose that $D\subseteq R$ is an extension of domains. By Zorn's Lemma, there exists a maximal (with respect to inclusion) subset $X$ of $R$ which is algebraically independent over $D$. By maximality, $R$ is algebraic over $D[X]$ (thus every integral domain is algebraic over a UFD; this can be seen by taking $D$ to be the prime subring of $R$). If $E$ and $F$ are the quotient fields of $D$ and $R$, respectively, then $X$ can be shown to be a transcendence basis for $F/E$ (that is, $X$ is maximal with respect to the property of being algebraically independent over $E$). The transcendence degree of $F$ over $E$ is the cardinality of a transcendence basis for $F/E$ (it can be shown that any two transcendence bases have the same cardinality). We denote the transcendence degree of $F$ over $E$ by $tr.deg(F/E)$. We remind that whenever $R\subsetneq T$ is an affine ring extension, i.e., $T$ is finitely generated as a ring over $R$ (in praticular, if $T$ is a finitely generated $R$-module), then by a natural use of Zorn Lemma, one can easily see that $T$ has a maximal subring $S$ which contains $R$. If $R$ is a proper subring of $T$ then $R$ is a maximal subring of $T$ if and only if for each $x\in T\setminus R$, we have $R[x]=T$. Finally we refer the reader to \cite{gilbok}, for standard definitions of Bezout domains, $QR$-domains and completely integrally closed domains.

\section{Existence of Conch Maximal Subring}
We begin this section by the following main result which is a generalization of \cite[Theorem 2.2]{azarang3}.

\begin{thm}\label{icmspsd}
Let $D$ be an integral domain with a prime element $p$ such that $\bigcap_{n=1}^\infty p^nD=0$. Assume that $D\subseteq R$ is a ring extension such that $p\in U(R)$. Then $R$ has an integrally closed maximal subring which contains $D$ and conches $\frac{1}{p}$.
\end{thm}
\begin{proof}
We may assume that $R$ is an integral domain which is algebraic over $D$. To see this, first note that $D\setminus \{0\}$ is a multiplicatively closed set in $R$, therefore $R$ has a prime ideal $Q$, with $Q\cap D=0$. In other words, $D$ can be embedded in $R/Q$, and clearly $p\in U(R/Q)$ (note that if $S/Q$ is an integrally closed maximal subring of $R/Q$ which contains the image of $D$ and conches $(p+Q)^{-1}$ in $R/Q$, then $S$ is an integrally closed maximal of $R$ which contains $D$ and conches $\frac{1}{p}$). Hence we may suppose that $R$ is an integral domain. If $R$ is not algebraic over $D$, then let $X$ be a transcendence base for $R$ over $D$. Clearly $p$ is prime in $D[X]$ and $\bigcap_{n=1}^\infty p^n D[X]=0$. Hence we can replace $D$ by $D[X]$. Thus assume that $D\subseteq R$ is an algebraic extension of integral domain with quotient fields $K\subseteq E$. It is clear that $E/K$ is an algebraic extension. Now we claim that $D_{(p)}[\frac{1}{p}]=K$, where $K$ is the quotient field of $D$. For proof let $x=\frac{a}{b}\in K$, where $a, b\in D$. Since $\bigcap_{n=1}^\infty p^nD=0$, we conclude that $b=p^nc$, where $c\notin pD$ and $n\geq 0$. Thus $\frac{a}{c}\in D_{(p)}$ and therefore $x=\frac{a/c}{p^n}\in D_{(p)}[\frac{1}{p}]$, i.e., $K=D_{(p)}[\frac{1}{p}]$. Next we prove that $S:=D_{(p)}$ is a maximal subring of $K$. To see this, let $x\in K\setminus S$. From $\bigcap_{n=1}^\infty p^nD=0$ and $x\notin S$, we easily see that $x=\frac{a}{p^nb}$, where $a,b\notin pD$ and $n\in\mathbb{N}$. Hence we infer that $\frac{1}{a}\in S$ and therefore $\frac{1}{p^n}\in S[x]$. Thus $\frac{1}{p}\in S[x]$, which immediately by the previous part implies that $S[x]=K$, i.e., $S$ is a maximal subring of $K$. Thus by \cite[Proposition 2.1]{azkarm3}, we conclude that $E$ has a maximal subring $V$ such that $V\cap K=S$. Hence $\frac{1}{p}\notin V$ and therefore we deduce that $V$ is an integrally closed maximal subring of $E$. Thus in the extension $R\subseteq E$ we have $U(R)\nsubseteq V$. Therefore by \cite[Theorem 2.19]{azkarm4}, we infer that $R$ has an integrally closed maximal subring $W$ which contains $V\cap R$ but $\frac{1}{p}\notin W$ i.e., $W$ is a maximal subring which contains $D$ and conches $\frac{1}{p}$ in $R$ (note that a conch maximal subring is integrally closed by the comment preceding Corollary \ref{corconch}).
\end{proof}

\begin{cor}\label{icmsipat}
Let $D$ be an atomic (or a completely integrally closed) domain and $R$ is a ring extension of $D$. If a prime element of $D$ is invertible in $R$, then $R$ has a conch maximal subring.
\end{cor}
\begin{proof}
Let $p$ be a prime element of $D$ which is invertible in $R$. We claim that $J:=\bigcap_{n=1}^\infty (p^n)=0$. If $D$ is completely integrally closed then we are done by \cite[Corollary 13.4]{gilbok}. Hence assume that $R$ is atomic and $J\neq 0$, then by \cite[Ex.5, Sec.1-1]{kap}, $J$ is a prime ideal of $R$ which is properly contained in $M:=(p)$. Now since $J\neq 0$ and $R$ is atomic domain, we immediately conclude that $J$ contains an irreducible element $q$ of $R$. Thus $q\in M=(p)$ and therefore $q=p$ which is absurd (see \cite{andchun} for more interesting result in arbitrary commutaive rings). Thus in any cases $J=0$ and hence we are done by the previous theorem.
\end{proof}

Fields which have (no) maximal subrings are completely determined in \cite{azkrm}. We need the following in sequel.

\begin{cor}\label{icmsnaf}
Let $K$ be a field which is not absolutely algebraic. Then $K$ has an integrally closed/conch maximal subring.
\end{cor}
\begin{proof}
If $Char(K)=0$ then $K$ contains $\mathbb{Z}$ and use the previous corollary. If $Char(K)=0$, then by our assumption, there exists $x\in K$, which is not algebraic over $\mathbb{Z}_p$. Therefore $K$ contains the atomic domain $\mathbb{Z}_p[x]$ and again use the previous corollary.
\end{proof}

Now we have the following main theorem.

\begin{thm}\label{uthm}
Let $T$ be a ring. Then either $T$ has a conch maximal subring or $U(R)=R\cap U(T)$ for each subring $R$ of $T$. In other words either $T$ has a conch maximal subring or each subring of $T$ is closed respect to taking the inverse.
\end{thm}
\begin{proof}
Suppose that $T$ has a subring $R$ with $U(R)\subsetneq R\cap U(T)$. Hence there exists $x\in R\setminus U(R)$ with $x\in U(T)$.  We show that $T$ has a maximal subring conches $x^{-1}$. First we claim that we may assume that $T$ is an integral domain. To see this, take $X:=\{1+xy\ |\ y\in R\}$, clearly $X$ is a multiplicatively closed subset of $R$ and $0\notin X$, for $x$ is not unit in $R$. Hence there exists a prime ideal $P$ of $R$ such that $P\cap X=\emptyset$. We may assume that $P$ is a minimal prime ideal of $R$ and therefore $x\notin P$ (note, $x$ is not a zero-divisor of $R$). Since $P$ is a minimal prime ideal of $R$, we conclude that there exists a (minimal) prime ideal $Q$ of $T$ such that $Q\cap R=P$ (see \cite[Exercise 1, P. 41]{kap}), therefore we can consider $R/P$ as a subring of $T/Q$. Now note that the image of $x$ is a unit of $T/Q$ but is not a unit of $R/P$; for otherwise there exists $y\in R$, such that $1-xy\in P$, i.e., $1+(-y)x\in P\cap X$ which is absurd. Thus the extension $R/P \subseteq T/Q$ is an extension of integral domains and $\bar{x}:=x+P$ is not a unit in $R/P$ but is a unit in $T/Q$. Clearly, if $S/Q$ is a maximal subring of $T/Q$ which conches $\bar{x}^{-1}$ in $T/Q$, then $S$ is a maximal subring of $T$ conches $x^{-1}$. Thus we may assume $T$ is an integral domain. Now we have two cases:\\
$({\bf 1})$ $Char(T)=p>0$, where $p$ is a prime number (note $T$ is an integral domain). Thus we infer that $x$ is not algebraic over the prime subring of $T$, for otherwise $x$ is integral over $\mathbb{Z}_p$, the prime subring of $R$ (or $T$). Therefore by $(u,u^{-1})$-Lemma, $x\in U(R)$ which is absurd. Therefore $x$ is not algebraic over $\mathbb{Z}_p$. Hence $D:=\mathbb{Z}_p[x]$ is an atomic domain which is contained in $T$ and the prime element $x$ of $D$ is invertible in $T$. Thus $T$ has a maximal subring that conches $x^{-1}$ by Corollary \ref{icmsipat}.\\
$({\bf 2})$ $Char(T)=0$. If $x$ is not algebraic over $\mathbb{Z}$, then similar to the proof of $(1)$ and considering the atomic domain $D:=\mathbb{Z}[x]$ in $T$, we conclude that $T$ has a maximal subring conches $x^{-1}$, by Corollary \ref{icmsipat}. Thus assume that $x$ is algebraic over $\mathbb{Z}$. Since $x^{-1}\notin R$, we conclude that $x^{-1}\not\in \mathbb{Z}[x]$, for $\mathbb{Z}[x]\subseteq R$. Thus by \cite[Theorem 30.9]{gilbok}, $dim(\mathbb{Z}[x])=1$  and therefore by \cite[Theorem 56]{kap}, the quotient field of $\mathbb{Z}[x]$, i.e., $\mathbb{Q}(x)$ has a valuation $V$ such that $xV\neq V$ (i.e., $x^{-1}\notin U(V)$). Since $\mathbb{Z}[x]$ is a one dimensional noetherian integral domain, then by Krull-Akizuki Theorem (\cite[Theorem 93]{kap}), we infer that $V$ is a noetherian one dimensional valuation domain. In other words, $V$ is a DVR and therefore $V$ is a UFD. Let $M=(\pi)$ denotes the maximal ideal of $V$. Thus $x=v\pi$ for some $v\in V$ (note, $V$ is a valuation and $x^{-1}\notin V$). Let $K$ be the quotient field of $T$, then $V\subseteq \mathbb{Q}(x)\subseteq K$. Therefore by Theorem \ref{icmspsd}, $K$ has a maximal subring $W$ such that $V\subseteq W$ and $\frac{1}{\pi}\notin W$. Thus $v\in W$ and $\pi\in N$, where $N$ is the maximal ideal of $W$. Therefore $x\in N$ and thus $U(T)\nsubseteq W$. Therefore by \cite[Theorem 2.19]{azkarm4}, $T$ has a maximal subring that conches $x^{-1}$.
\end{proof}

In other words, the previous theorem state that if $R\subseteq T$ is a ring extension and a non invertible element $x$ of $R$, is invertible in $T$, then $T$ has a maximal subring. In particular, if a ring $T$ has no maximal subring, then $xR\neq R$ implies $xT\neq T$, for each subring $R$ of $T$ and $x\in R$. In other words either a ring $T$ has a maximal subring or each principal proper ideal of any subring of $T$, survives in $T$. The previous theorem has several conclusion as follows.

\begin{cor}\label{cu1}
Let $T$ be a ring, then $T$ has a conch maximal subring if and only if $U(T)$ is not integral over $\Z$. In particular, if $T$ has no conch maximal subring then $\Z[x]=\Z[x^{-1}]$, for each $x\in U(T)$.
\end{cor}
\begin{proof}
If $T$ has a subring $R$ which conches $x^{-1}$ in $T$, then by $(u,u^{-1})$-Lemma $x^{-1}$ is not integral over $R$ and therefore $x^{-1}$ is not integral over $\Z$. Conversely, assume that $T$ has no conch maximal subring, then for each $x\in U(T)$, by Theorem \ref{uthm}, we conclude that $U(\Z[x])=\Z[x]\cap U(T)$. Therefore $x^{-1}\in \Z[x]$. Thus by $(u,u^{-1})$-Lemma, we deduce that $x^{-1}$ is integral over $\Z$. Thus $U(T)$ is integral over $\Z$. Hence if $T$ has no conch maximal subring, then for each $x\in U(T)$ we obtain that $x^{-1}\in \Z[x]$ and $x\in \Z[x^{-1}]$. Therefore $\Z[x]=\Z[x^{-1}]$.
\end{proof}

In other words the previous corollary state that a ring $T$ has a subring $S$ which conches $x^{-1}$ in $T$ if and only if $T$ has a maximal subring which conches $x^{-1}$ in $T$.

\begin{cor}\label{cu2}
Let $T$ be a ring with nonzero characteristic. Then either $T$ has a conch maximal subring or each element of the group $U(T)$ has finite order.
\end{cor}
\begin{proof}
Assume that $Char(T)=n$ and therefore $\mathbb{Z}_n$ is the prime subring of $T$. Thus if $T$ has no maximal subring, then for each $x\in U(T)$, $x$ is integral over $\mathbb{Z}_n$.
Hence, $\mathbb{Z}_n[x]$ is a finitely generated $\mathbb{Z}_n$-module. Therefore $\mathbb{Z}_n[x]$ is finite. This immediately implies that $x$ has finite order in $U(T)$.
\end{proof}

\begin{cor}\label{cu3}
Let $T$ be a ring without maximal subring and $S\subseteq R$ be subrings of $T$. Then $U(S)=S\cap U(R)$.
\end{cor}
\begin{proof}
Since $T$ has no maximal subring, then by Theorem \ref{uthm}, $U(R)=R\cap U(T)$ and $U(S)=S\cap U(T)$. Thus $U(S)=S\cap U(T)= (S\cap R)\cap U(T)=S\cap (R\cap U(T))=S\cap U(R)$.
\end{proof}

\begin{cor}\label{cu4}
Let $R$ be a ring and $X$ be a multiplicatively closed subset of regular (i.e., non zero-divisors) of $R$. Then either $R_X$ has a maximal subring or $R_X=R$, i.e., $X\subseteq U(R)$. In particular, if $R$ is an integral domain, then each proper quotient overring of $R$ (i.e., $R\subsetneq R_X$) has a maximal subring.
\end{cor}
\begin{proof}
It suffices to take $T=R_X$ in Theorem \ref{uthm}.
\end{proof}

\begin{cor}\label{cumcasg}
Let $R$ be an integral domain and $X$ be a multiplicatively closed subset of $R$ which is not contained in $U(R)$. Then there exists a proper subring $S$ of $R_X$ and $a\in X$
such that $R_X=S_a$.
\end{cor}

An integral domain $R$ is called a $QR$-domain, if each overring of $R$ is of the form $R_X$ for some multiplicatively closed subset $X$ of $R$. It is well-known that each Bezout domain is a $QR$-domain.

\begin{cor}
Let $R$ be a $QR$-domain (which is not absolutely algebraic field). Then each proper overring of $R$ has a maximal subring.
\end{cor}

\begin{cor}\label{cu41}
Let $R$ be a ring. Then either $Q(R)$ has a maximal subring or $Q(R)=R$.
\end{cor}
\begin{proof}
Let $X=R\setminus Zd(R)$, therefore $Q(R)=R_X$. If for some $x\in X$, we have $x^{-1}\notin R$, then by Theorem \ref{uthm}, $Q(R)$ has a maximal subring. Hence if $Q(R)$ has no maximal subring, then for each $x\in X$ we conclude that $x^{-1}\in R$ and therefore $R=Q(R)$.
\end{proof}

\begin{cor}\label{cu5}
Let $R$ be a noetherian ring. Then either $Q(R)$ has a maximal subring or $R$ is a countable artinian ring with nonzero characteristic which is integral over its prime subring.
\end{cor}
\begin{proof}
Assume that $Q(R)$ has no maximal subring. Therefore by Corollary \ref{cu41}, $R=Q(R)$. Now note that since $R$ is noetherian ring, then $Ass(o)$ is finite and $Zd(R)=\bigcup_{P\in Ass(0)} P$. This immediately implies that $Q(R)$ is a semilocal ring. Thus by \cite[Proposition 3.13]{azkarm4}, we infer that $R=Q(R)$ is an artinian ring with nonzero characteristic which is integral over $\Z$.
\end{proof}

We remind that an integral domain $R$ satisfies the ascending chain condition for principal ideals (ACCP) if there does not exist an infinite strictly ascending chain of principal (integral) ideals of $R$, see \cite{andzaf1}.

\begin{cor}\label{cu6}
Let $T$ be an integral domain satisfies ACCP. Then either $T$ has a conch maximal subring or each subring of $T$ satisfies ACCP. In particular, if $T$ has no maximal subring, then each subring $R$ of $T$ is atomic and if $(a_1)\supset (a_2)\supset \cdots$ is an infinite strictly descending chain of principal ideals of $R$, then $\bigcap_{i=1}^\infty (a_n)=0$. Consequently, for each $a\in R\setminus U(R)$ we have $\bigcap_{i=1}^\infty (a^n)=0$
\end{cor}
\begin{proof}
Assume that $T$ has no maximal subring, then by Theorem \ref{uthm}, we conclude that $U(R)=R\cap U(T)$. Therefore by \cite[Proposition 2.1 and Corollary]{grm}, $R$ has ACCP. The final part is evident for it is well-known that domains with ACCP are atomic.
\end{proof}

We remind that an integral domain $R$ is called a bounded factorization domain (BFD) if $R$ is atomic and for each nonzero nonunit of $R$ there is a bound on the length of factorization into product of irreducible elements, i.e., for each nonzero nonunit $x$ of $R$, there exists a positive integer $N(x)$ such that whenever $x=x_1\cdots x_n$ as a product of irreducible elements of $R$, then $n\leq N(x)$. One can easily see that a BFD satisfies ACCP, but the converse does not hold, see \cite[Example 2.1]{andzaf1}. Also a Noetherian domain or a Krull domain is a BFD, see \cite[Proposition 2.2]{andzaf1}.

\begin{cor}\label{cu61}
Let $T$ be a BFD. Then either $T$ has a maximal subring or each subring of $T$ is a BFD.
\end{cor}
\begin{proof}
Assume that $T$ has no maximal subring, then by Theorem \ref{uthm}, we conclude that $U(R)=R\cap U(T)$. Therefore by the comment preceding \cite[Proposition 2.6]{andzaf1}, $R$ is a BFD.
\end{proof}

\begin{cor}\label{cu7}
Let $R$ be a GCD-domain and $T$ a ring extension of $R$. Assume that there exist $a,b\in R$ such that $gcd(a,b)=1$, $ab\notin U(R)$ and $\frac{a}{b}\in U(T)$. Then $T$ has a maximal subring. In particular, if $T$ has no maximal subring, then $U(T)\cap K=U(R)$, where $K$ is the quotient field of $R$.
\end{cor}
\begin{proof}
Let $x=\frac{a}{b}$. If $T$ has no maximal subring, then $x$ (resp. $x^{-1}$) is integral over the prime subring of $T$ and therefore $x$ (resp. $x^{-1}$) is integral over $R$. Thus $x$ and $x^{-1}$ are in $R$, for a $GCD$-domain is integrally closed, i.e., $b|a$ and $a|b$ in $R$. Therefore $a$ and $b$ are units for $gcd(a,b)=1$, which is absurd. Thus $T$ has a maximal subring. The final part is evident.
\end{proof}

\begin{cor}\label{cu8}
Let $T$ be a ring such that the prime subring $\Z$ of $T$ is integrally closed in $T$. Then either $T$ has a maximal subring or $U(T)=U(\Z)$. In particular, if $T$ has no maximal subring, then $U(T)$ is finite.
\end{cor}
\begin{proof}
By Theorem \ref{uthm}, if $T$ has no maximal subring, then $U(T)$ is integral over $\Z$ and therefore $U(T)\subseteq \Z$, for $\Z$ is integrally closed in $T$. This immediately implies that $U(T)=U(\Z)$ and therefore $U(T)$ is finite.
\end{proof}

\begin{prop}\label{pu}
Let $T$ be a ring and $R:=\Z[U(T)]$. The following are equivalent.
\begin{enumerate}
\item $R$ has a proper subring which is integrally closed in $R$.
\item $R$ has an integrally closed maximal subring.
\item $T$ has a maximal subring $A$ and there exists $u^{-1}\in U(T)\setminus A$ with $u\in A$ (i.e., $T$ has a conch maximal subring).
\item $U(T)$ is not integral over $\Z$.
\end{enumerate}
\end{prop}
\begin{proof}
$(1)\Longrightarrow (2)$ If $S$ is a proper integrally closed subring of $R$, then clearly $U(T)$ is not integral over $S$. Therefore $U(T)=U(R)$ is not integral over $\Z$. Thus by Theorem \ref{uthm}, $R$ has an integrally closed maximal subring. $(2)\Longrightarrow (3)$ If $A$ is an integrally closed maximal subring of $R$, then $U(R)=U(T)$ is not integral over $A$. Therefore $U(T)$ is not integral over $\Z$ and hence we are done by Theorem \ref{uthm}. $(3)\Longrightarrow (4)$ $U(T)$ is not integral over $A$ and therefore is not integral over $\Z$. $(4)\Longrightarrow (1)$ By Theorem \ref{uthm}, $R$ has an integrally closed maximal subring (which is proper).
\end{proof}

\begin{prop}\label{fgugcs}
Let $T$ be a ring without maximal subrings and $U(T)$ is finitely generated group. Then
\begin{enumerate}
\item If $Char(T)=n$, then $\mathbb{Z}_n[U(T)]$ is finite. In particular, $U(T)$ is finite.
\item If $Char(T)=0$, then $\mathbb{Z}[U(T)]$ is a finitely generated $\mathbb{Z}$-module.
\item If $Char(T)=0$ and $T$ is a residually finite ring, then $T=\mathbb{Z}$.
\end{enumerate}
\end{prop}
\begin{proof}
First note that if $U(T)=<u_1,\ldots,u_n>$, then by our assumption and Theorem \ref{uthm}, for each $i$, $u_i^{-1}$ is integral over $\Z$ and therefore $u_i^{-1}\in \Z[u_i]\subseteq \Z[u_1,\ldots,u_n]$. Therefore $\Z[U(T)]=\Z[u_1,\ldots,u_n]$ is integral over $\Z$. Thus $\Z[U(T)]$ is finitely generated $\Z$-module. In particular, if $Char(R)=n>0$ (i.e., $\Z=\mathbb{Z}_n$), then $\Z[U(T)]$ is finite. Hence $(1)$ and $(2)$ hold. For $(3)$, By \cite[Theorem 2.29]{azarang3}, $T=\Z[U(T)]$. Therefore by the first part of the proof $T=\Z[u_1,\ldots, u_n]$. Hence if $U(T)\neq U(\Z)$, then $T$ is finitely generated as a ring over $\Z$ and thus $T$ has a maximal subring which is absurd. Therefore $U(T)=U(\Z)$ and thus $T=\Z$, hence we are done.
\end{proof}

A ring $R$ is called clean, if each nonzero element of $R$ is a sum of a unit and an idempotent of $R$. Now the following is in order.

\begin{prop}
Let $R$ be a clean ring. Then either $R$ has a maximal subring or $R$ has nonzero characteristic and is integral over $\Z$.
\end{prop}
\begin{proof}
Assume that $R$ has no maximal subring. Then by Corollary \ref{cu1}, $U(R)$ is integral over $\Z$. Also note that clearly each idempotent of $R$ is integral over $\Z$. Therefore by our assumption each element of $R$ is a sum of two integral element over $\Z$. Hence $R$ is integral over $\Z$. Finally, if $Char(R)=0$, then $dim(R)=1$ for $R$ is integral over $\mathbb{Z}$. Let $P$ be prime ideal of $R$ which is not maximal, then $R/P$ is a local integral domain (note, in clean ring each prime ideal is contained in a unique maximal ideal). Therefore by \cite[Corollary 2.24]{azkarm4}, $R/P$ and therefore $R$ has a maximal subring which is absurd. Hence $R$ has a nonzero characteristic and hence we are done.
\end{proof}

\begin{exm}
\begin{enumerate}
\item There exists a ring $R$ which has a maximal subring, $R$ is integral over $\Z$ and moreover $U(R)$ is a finitely generated. To see this, let $K$ be a finite field extension of $\mathbb{Q}$ of degree $n\geq 2$ and $R$ be the integral closure of $\mathbb{Z}$ in $K$. Then it is well-known that $R$ is a free $\mathbb{Z}$-module with rank $n$ and clearly $R$ (and therefore $U(R)$ are integral over $\mathbb{Z}$) and by Dirichlet's unit theorem $U(R)$ is finitely generated. Finally note that since $R\neq\mathbb{Z}$ is a finitely generated algebra, $R$ has a maximal subring (clearly $R$ has no conch maximal subring).
\item There exists a ring $R$ which has a maximal subring, $R$ is not integral (algebraic) over its prime subring and moreover $U(S)=S\cap U(R)$ for each subring $S$ of $R$. To see this, let $p$ be a prime number and $R=\mathbb{Z}_p[X]$ be the polynomial ring over $\mathbb{Z}_p$. Then clearly $R$ has a maximal subring, and $R$ is not algebraic over $\mathbb{Z}_p$. Finally note that for each subring $S$ of $R$, $\mathbb{Z}_p\subseteq S$ and therefore $U(S)=U(R)$.
\item For each infinite cardinal number $\alpha$, there exists a ring $R$ with $|R|=|U(R)|=\alpha$ and $R$ has no maximal subring. In particular, if $\alpha>\aleph_0$, then $U(R)$ is not finitely generated. To see this note that if $K$ is a field with zero characteristic and $|K|=\alpha$, then the idealization $R=\mathbb{Z}(+)K$ has no maximal subring by \cite[Example 3.19]{azkarm3} (for $K$ has no maximal $\mathbb{Z}$-submodule). It is clear that $1(+)K\subseteq U(R)$ and therefore $|U(R)|=\alpha=|R|$.
\item There exists a conch subring of a ring $T$ which is not a maximal subring of $T$. To see this, let $V$ be a valuation domain with $dim(V)=n\geq 2$ and quotient field $K$. Suppose $P$ and $Q$ are prime ideal of $V$ with $ht(Q)=n$ and $ht(P)=n-1$. Let $x\in Q\setminus P$. Then $V$ conches $x^{-1}$ in $K$ but $V\subsetneq V_P\subsetneq K$, i.e., $V$ is not a maximal subring of $K$.
\end{enumerate}
\end{exm}

\begin{prop}
Let $T$ be a ring with $|T|>2^{2^{\aleph_0}}$. Let $R$ be the integral closure of $\Z$ in $T$ and $X$ be a generator set for the group $U(T)$. Then either $T$ has a maximal subring or $|T|=|R|=|X|$. In particular, if $S$ is an integrally closed subring of $T$, then $U(T)=U(S)$ and $|S|=|T|$.
\end{prop}
\begin{proof}
First note that by \cite[Corollary 2.4]{azkarm3}, either $T$ has a maximal subring or $|U(T)|=|T|$. Now assume that $T$ has no maximal subring, therefore by Corollary \ref{cu1}, $U(T)$ is integral over $\Z$. Hence $U(T)\subseteq R$ and therefore $|R|=|T|$. Next, we show $|X|=|T|$. If $X$ is finite, then $U(T)$ is finitely generated and therefore $U(T)$ is countable which is absurd. Hence $X$ is infinite. Similar to the proof of Proposition \ref{fgugcs}, $\Z[U(T)]=\Z[X]$. One can easily see that $|\Z[X]|=|X|$ and therefore $|U(T)|=|X|$. The final part is evident.
\end{proof}

In the next result we proved that whenever $R$ is a local integral domain which is an integrally closed maximal subring of a ring $T$, then $R$ is a conch subring of $T$.

\begin{prop}
Let $R$ be an integral domain which is an integrally closed  maximal of a ring $T$. Then
\begin{enumerate}
\item If $R$ is a local ring, then $R$ is a conch subring of $T$.
\item If $M$ is the crucial maximal of $R\subseteq T$, then $R_M$ is a conch subring of $T_M$.
\end{enumerate}
\end{prop}
\begin{proof}
First note that by \cite[Theorem 10]{modica}, $T$ is an integral domain. Let $y\in T\setminus R$. Then by maximality of $R$ we have $R[y]=T$. Now by a similar proof
\cite[Lemma 19.14]{gilbok}, we conclude that $y^{-1}\in R$ and therefore $R$ conch $y$ in $T$. This proves $(1)$. $(2)$ is trivial by $(1)$ and $(4)$ of Theorem \ref{pret1}.
\end{proof}

In \cite[Corollaries 3.6 and 3.7]{azkarm4}, we prove that if $T$ is a reduced ring with $J(T)\neq 0$ or $|T|>2^{2^{\aleph_0}}$, then $T$ has a maximal subring. Now the following is in order.

\begin{thm}\label{brdr1d}
Let $T$ be a reduced ring with $|T|=2^{2^{\aleph_0}}$. Then either $T$ has a maximal subring or $|Max(T)|=2^{\aleph_0}$ and the intersection of each countable family of maximal ideals of $T$ is nonzero.
\end{thm}
\begin{proof}
Assume that $T$ has no maximal subring, then by \cite[Corollary 3.6]{azkarm4}, $J(T)=0$. Therefore $T$ can be embedded in $\prod_{M\in Max(T)} \frac{T}{M}$. Since $T$ has no maximal subring, then by \cite[Propossition 2.6]{azkrm}, we conclude that $|Max(T)|\leq 2^{\aleph_0}$ and by \cite[Corollary 1.3]{azkrm} for each maximal ideal $M$ of $T$ we have $|R/M|\leq\aleph_0$. Now if $|Max(T)|<2^{\aleph_0}$, then we deduce that
$$|T|\leq|\prod_{M\in Max(T)} \frac{T}{M}|\leq \aleph_0^{\aleph_0}$$
which is absurd. Thus we infer that $|Max(T)|=2^{\aleph_0}$. The proof of the final part is similar.
\end{proof}

Now we have the following.

\begin{thm}\label{rdlcd}
Let $T$ be an integral domain with $|T|=2^{2^{\aleph_0}}$ and $dim(T)=1$. Then $T$ has a maximal subring.
\end{thm}
\begin{proof}
First note that if $Char(T)=p$ is a prime number, then $T$ is not algebraic over $\mathbb{Z}_p$; thus there exists $x\in T$  which is not algebraic over the prime subring of $T$. Now if $T$ has zero characteristic, then we take $D=\mathbb{Z}$ and if $Char(T)=p>0$ (where $p$ is a prime number), we take $D=\mathbb{Z}_p[x]$. In any cases $D$ is a PID with infinitely many non-associate prime elements. Let $Irr(D)=\{q_1, q_2,\ldots\}$. We have two cases. If there exists $i$ such that $q_i\in U(T)$, then by Theorem \ref{icmspsd}, $T$ has a maximal subring. Hence suppose that for each $i$, $q_i$ is not invertible in $T$ and therefore $T$ has a maximal ideal $M_i$ such that $q_i\in M_i$. Now by Theorem \ref{brdr1d}, $N:=\bigcap_{i=1}^\infty M_i\neq 0$. We claim that $N\cap D=0$. To see this, let $q\in N\cap D$, then for each $i$ we have $q\in M_i\cap D=q_iD$ and therefore $q=0$. Thus $N\cap X=\emptyset$, where $X:=D\setminus \{0\}$ is a multiplicatively closed subset in $T$. Therefore $T$ has a prime ideal $P$ with $N\subseteq P$ and $P\cap X=\emptyset$. Since $dim(T)=1$, we conclude that $P$ is a maximal ideal of $T$. From $D\cap P=0$ we deduce that the field $T/P$ contains a copy of $D$. Therefore $T/P$ is not absolutely algebraic field and hence by Corollary \ref{icmsnaf}, $T/P$ has a maximal subring. Thus $T$ has a maximal subring and we are done.
\end{proof}

\begin{prop}
Let $T$ be an integral domain with $|T|=2^{2^{\aleph_0}}$ and $dim(T)=2$. Let $H_1$ be the set of all height one prime ideals of $T$. Then either $T$ has a maximal subring or $\bigcap_{P\in H_1} P=0$, $|H_1|\geq 2^{\aleph_0}$ and $|R/P|\leq 2^{\aleph_0}$ for each $P\in H_1$.
\end{prop}
\begin{proof}
Assume that $T$ has no maximal subring. Therefore by \cite[Corollary 2.24]{azkarm4}, $J(T)=0$. Since $dim(T)$ is finite we conclude that each maximal ideal of $T$ contains a height
one prime ideal and clearly each hight one prime ideal is contained in a maximal ideal of $T$. Thus $\bigcap_{P\in H_1}P\subseteq J(T)=0$. Hence $T$ embeds in $\prod_{P\in H_1}R/P$. Therefore, if $|H_1|<2^{\aleph_0}$, then there exists $P\in H_1$ such that $|R/P|\geq 2^{2^{\aleph_0}}$. Now either $R/P$ is a field or $dim(R/P)=1$. In the former case $R/P$ has a maximal subring by Corollary \ref{icmsnaf} and in the later case $R/P$ has a maximal subring by Theorem \ref{rdlcd}, which is a contradiction in any cases. Thus $|H_1|\geq 2^{\aleph_0}$.
\end{proof}

\begin{prop}
Let $K$ be a field, $R$ a ring extension of $K$, $x\in R\setminus (U(R)\cup Zd(R))$. If $R=K+Rx$, then $R$ has a maximal subring.
\end{prop}
\begin{proof}
First we show that $K+Rx^2$ is a proper subring of $R$. It is clear that $K+Rx^2$ is a subring of $R$. Now if $R=K+Rx^2$, then there exist $a\in K$ and $r\in R$ such that $x=a+rx^2$. If $a\neq 0$, then we conclude that $x(1-rx)=a\in U(R)$ which is absurd. Hence $a=0$ and therefore $x=rx^2$, and since $x$ is not a zero-divisor we deduce that $1=rx$
which is impossible by our assumption. Thus $K+Rx^2$ is a proper subring of $R$. Now note that $R=K+Rx=K+(K+Rx)x=K+Kx+Rx^2$. Therefore $(K+Rx^2)[x]=R$. Thus $R$ has a maximal subring. \end{proof}

\begin{prop}
Let $R$ be an uncountable PID, then any ring extension of $R$ has a maximal subring.
\end{prop}
\begin{proof}
Let $T$ be a ring extension of $R$. If $U(R)$ is uncountable then we are done by Theorem \ref{uthm}, for in this case $U(R)$ is not algebraic over $\Z$ and therefore $U(T)$ is not integral over $\Z$. Hence assume that $U(R)$ is countable. Thus by the proof of \cite[Theorem 3.1]{azarang3}, we infer that $R$ has a prime elements $q$ such that $U(R/(q))$ is not algebraic over the prime subring of $R/(q)$. Now we have two cases. If $qT=T$, then by Theorem \ref{uthm}, $T$ has a maximal subring. Otherwise $qT$ is a proper ideal of $T$ and clearly $qT\cap R=qR$, for $qR$ is a maximal ideal of $R$. Thus $T/qT$ contains a copy of $R/qR$. Hence $U(T/qT)$ contains a copy of $U(R/qR)$. Therefore $U(T/qT)$ is not integral over the prime subring of $T/qT$. Therefore $T/qT$ has a maximal subring by Theorem \ref{uthm}. Thus $T$ has a maximal subring.
\end{proof}

\begin{prop}\label{atidlcpm}
Let $R$ be an atomic (or a completely integrally closed) domain with $|R|=2^{2^{\aleph_0}}$. If $M$ is a principal maximal ideal of $R$, then $|R/M|=|R|$ and therefore $R$ has a maximal subring.
\end{prop}
\begin{proof}
If $M=0$, then $R$ is a field and therefore we are done by Corrollary \ref{icmsnaf}. Thus assume that $M=(p)$ is a nonzero principal maximal ideal of $R$. Hence $p$ is a prime element of $R$. Similar to the proof of Corollary \ref{icmsipat}, we conclude that $\bigcap_{n=1}^\infty (p^n)=0$. Therefore $R$ can be embedded in $\prod_{n=1}^\infty R/(p^n)$. Thus there exists $n$ such that $|R/(p^n)|=2^{2^{\aleph_0}}$. Hence by \cite[Lemma 2.8]{azkarm3}, $|R/(p)|=2^{2^{\aleph_0}}$ and therefore $R/(p)$ has a maximal subring, by Corollary \ref{icmsnaf}. Thus $R$ has a maximal subring.
\end{proof}

\section{Conductor of Integrally Closed Maximal Subring}

In this section we are interested to show that if $K$ is an algebraic closed field, which is not algebraic over its prime subfield, and $R$ is affine ring over $K$, then for each prime ideal $P$ of $R$ with $ht(P)\geq dim(R)-1$, there exists an integrally closed maximal subring $S$ of $R$ which is integrally closed in $R$ and $(S:R)=P$. First we begin by the following lemma for arbitrary ring $R$.

\begin{lem}\label{mscpnm}
Let $R$ be a maximal subring of a ring $T$ with $(R:T)\in Spec(T)\setminus Max(T)$. Then $R$ is integrally closed in $T$.
\end{lem}
\begin{proof}
If $T$ is integral over $R$, then by \cite[Theorem 2.8]{gilmer3}, $P:=(R:T)$ satisfies exactly one of the following:
\begin{enumerate}
\item $P$ is a maximal ideal of $T$.
\item There exists a maximal ideal $M$ of $T$ such that $M^2\subseteq P\subseteq M$. Therefore $M=P$, for $P$ is prime.
\item There exist distinct maximal ideals $M_1$ and $M_2$ of $T$ such that $P=M_1\cap M_2$. Thus either $P=M_1$ or $P=M_2$, for $P$ is prime.
\end{enumerate}
Hence in any case we conclude that $P$ is a maximal ideal of $T$, which is impossible. Thus $R$ is integrally closed in $T$.
\end{proof}

We remind that if $V$ is an integral domain with quotient field $K\neq V$, then one can easily see that the following are equivalent:
\begin{enumerate}
\item $V$ is a maximal subring of $K$
\item $V$ is a one dimensional valuation domain.
\item $V$ is a real (archimedean) valuation domain.
\item $V$ is a completely integrally closed valuation domain.
\end{enumerate}

Now the following is in order.

\begin{lem}\label{cicmstf}
Let $K$ be a field and $T$ be a ring extension of $K$. If $V$ is a maximal subring of $T$ which is integrally closed in $T$ and $K\nsubseteq T$, then $V\cap K$ is a maximal subring of $K$. In particular, $V\cap K$ is one dimensional valuation domain.
\end{lem}
\begin{proof}
First note that since $V$ is integrally closed in $T$, then for each $u\in U(T)$ either $u\in V$ or $u^{-1}\in V$, by $(3)$ of Theorem \ref{pret1}. Since $K\nsubseteq V$, we infer that $K\cap V$ is a proper subring of $K$. Now for maximality of $K\cap V$ in $K$, it suffices to show that $(K\cap V)[\alpha]=K$, for each $\alpha\in K\setminus (K\cap V)$. By the first part of the proof note that $\alpha^{-1}\in V$. Since $V$ is a maximal subring of $T$ we conclude that $V[\alpha]=T$. Now suppose $\beta\in K$, thus $\beta\in T=V[\alpha]$, which implies that $\beta=v_0+v_1\alpha+\cdots+v_n\alpha^n$ for some $v_i \in V$. Therefore $\beta\alpha^{-n}=v_0\alpha^{-n}+\cdots+v_n=v\in V$, for $\alpha^{-1}\in V$. Finally note that since $K$ is a field and $\alpha,\beta\in K$ we have $\beta\alpha^{-n}=v\in K\cap V$ and therefore $\beta=v\alpha^n\in (K\cap V)[\alpha]$. Hence $K\cap V$ is a non field maximal subring of $K$ and therefore is a one dimensional valuation domain.
\end{proof}

Let $K$ be a field, $X_1,\ldots, X_n$ are indeterminates over $K$, $T=K[X_1,\ldots, X_n]$ and $Q\in Spec(T)$. Now a natural and stronger question arises from the existence of integrally closed maximal subrings is as follows: Does there exists an integrally closed maximal subring of $T$ with conductor $Q$? If $n=1$ and $K$ is absolutely algebraic field (i.e., $K$ is algebraic over $\mathbb{Z}_p$ for some prime $p$), then $K[X]$ has no integrally closed maximal subring, by \cite[Lemma 4.6]{azarang6} (the only integrally closed subrings of $K[X]$ are $K$ and $K[X]$). Therefore the answer to the question is not positive in general. But we show that if $K$ is algebraically closed field which is not algebraic over its prime subring and $n-1\leq ht(Q)\leq n$, then $T$ has an integrally closed maximal subring with conductor $Q$. We need some observation. Let $k$ be a field, then the Hahn field (or the field of generalized formal power series) over $k$ with exponents in $\mathbb{Q}$ (for general definition see \cite[Section 2.8]{efrat}) is denoted by $k[[t^{\mathbb{Q}}]]$, is the set of all formal power series $\alpha=\sum_{s\in\mathbb{Q}}\alpha_s t^s$ with $\alpha_s\in k$ and $supp(\alpha);=\{s\in\mathbb{Q}\ |\ \alpha_s\neq 0\}$ is a well-ordered subset of $\mathbb{Q}$. It is clear that if $t$ is an indeterminate over $k$, then $k\subseteq k[t]\subseteq k[[t^\mathbb{Q}]]$ (also $k[[t]]\subseteq k((t))\subseteq k[[t^\mathbb{Q}]]$). A natural valuation on $k[[t^{\mathbb{Q}}]]$ onto $\mathbb{Q}$, which send $\alpha$ to $min(supp(\alpha))$ is an archimedean valuation and hence its valuation ring $V$ is a maximal subring of $k[[t^Q]]$. It is well-known that if $k$ is algebraically closed, then $k[[t^\mathbb{Q}]]$ is algebraically closed and therefore it contains a copy of algebraic closure of $k(t)$, say $K$. Therefore by Corollary \ref{cicmstf}, $K\cap V$ is a maximal subring of $K$. In \cite[Proposition 5.17]{mau}, the authors characterized exactly maximal subrings of $K[X]$ which contain $(K\cap V)[X]$. These maximal subrings are integrally closed in $K[X]$ and the conductor of them is either $0$ or is of the form $(X-a)K[X]$ for some $a\in K$ (note $K$ is algebraically closed). Therefore in the later case these maximal subrings are of the form $(V\cap K)+(X-a)K[X]$. Now we have the following immediate corollary.

\begin{cor}\label{ckkx}
Let $K$ be an algebraically closed field which is not absolutely algebraic. Then $K[X]$ has an integrally closed maximal subring with zero conductor.
\end{cor}
\begin{proof}
By the above observation, it suffices to show that there exists an algebraically closed field $k$ and an indeterminate $t$ over $k$ such that $K$ is equal to the algebraic closure of $k(t)$ in $k[[t^{\mathbb{Q}}]]$. Let $F$ be the prime subfield of $K$ and $S$ be a transcendence basis for $K$ over $F$. Since $K$ is not algebraic over $F$ we infer that $S\neq\emptyset$. Let $t\in S$ and $S'=S\setminus\{t\}$, then clearly $E=F(S')$ is contained in $K$, $t$ is not algebraic over $E$ and $K$ is algebraic over $E(t)$. Now let $k$ be the algebraic closure of $E$ in $K$, then it is obvious that $t$ is not algebraic over $k$, $k$ is algebraically closed and $K$ is algebraic over $k(t)$. Thus $K$ is the algebraic closure of $k(t)$. Therefore the algebraic closure of $k(t)$ in $k[[t^{Q}]]$, say $L$ is isomorphic (as field) to $K$. Thus by the above observation $L[X]$ has an integrally closed maximal subring with zero conductor and hence the same is true for $K[X]$.
\end{proof}

We need the following result which is a lying-over property of conductors of integrally closed maximal subrings in integral extensions.

\begin{thm}\label{exinicsc}
Let $R\subseteq T$ be an integral extension of rings, $Q\in Spec(T)$ and $P=Q\cap R$. Assume that $R$ has an integrally closed maximal subring $S$ with $(S:R)=P$ and $U(R/P)\nsubseteq S/P$. Then $T$ has an integrally closed maximal subring $V$ with $(V:T)=Q$.
\end{thm}
\begin{proof}
It is clear that $A:=R/P\subseteq B:=T/Q$ is an integral extension and $C:=S/P$ is an integrally closed maximal subring of $A$ with $(C:A)=0$ and $U(A)\nsubseteq C$. Hence by $(3)$ of Theorem \ref{pret1}, assume that $\alpha\in U(A)$ such that $\alpha\in C$ but $\alpha^{-1}\notin C$. Thus $C[\alpha^{-1}]=A$, by maximality of $C$. Let $D$ be the integral closure of $C$ in $B$, then $\alpha^{-1}\notin D$ and one can easily see that $D[\alpha^{-1}]=B$ (similar to the proof of Lemma \ref{cicmstf}). Therefore $B$ has a maximal subring $E$ which contains $D$ but $\alpha^{-1}\notin E$ (see \cite[Proposition 2.1]{azkarm3}). Clearly, $E$ is integrally closed in $B$ and $E\cap A=C$. Let $V$ be a subring of $T$ such that $E=V/Q$ and $Q_1=(V:T)$. Thus $P\subseteq Q_1\cap R$, for $Q\subseteq Q_1$. Now let $x\in Q_1\cap R$, then $x+Q \in (V/Q)\cap (R/P)=S/P$, i.e., $x\in S$. Therefore $Q_1\cap R\subseteq S$. Thus $P\subseteq Q_1\cap R\subseteq (S:R)=P$ which immediately implies that $Q_1\cap R=P$. Now since $R\subseteq T$ is an integral extension (and therefore INC holds) and $Q\subseteq Q_1$ are primes ideals of $T$ with a same contraction in $R$, we conclude that $Q=Q_1$. Therefore $(V:T)=Q$ and we are done.
\end{proof}

The following is the main result in this section.

\begin{thm}\label{icwtphtn}
Let $K$ be an algebraically closed field which is not absolutely algebraic and $X_1,\ldots, X_n$ be indeterminates over $K$. Then for each prime ideal $Q$ of $T=K[X_1,\ldots,X_n]$ with $n-1\leq ht(Q)\leq n$, there exists an integrally closed maximal subring $S$ of $R$ with $(S:R)=Q$.
\end{thm}
\begin{proof}
We have two cases. First assume that $n=1$. If $Q=0$ then we are done by Corollary \ref{ckkx}, hence suppose that $Q$ is a maximal ideal of $T$ and therefore $T/Q\cong K$. By Corollary \ref{icmsnaf}, $T/Q$ has an integrally closed maximal subring $V/Q$ and $(V/Q: T/Q)=0$. Thus $V$ is an integrally closed maximal subring of $T$ which contains $Q$ and therefore $Q=(V:T)$, by maximality of $Q$. Now suppose that $n\geq 2$. If $ht(Q)=n$, then $Q$ is a maximal ideal of $T$ and similar to the case $n=1$ we are done. Hence assume that $ht(Q)=n-1$. Thus $dim(\frac{T}{P})=1$. Therefore by Noether's Normalization Theorem (see \cite[Theorem A1 P.221 or Theorem 13.3]{eisnb}), there exist $Y_1,\ldots, Y_n$ in $T$ such that $T$ is integral over $R:=K[Y_1,\ldots,Y_n]$ and $P:=Q\cap R=(Y_{2},\ldots, Y_{n})$. Thus $T/Q$ is integral over $K[Y_1]$. Now by Corollary \ref{ckkx}, $K[Y_1]$ has an integrally closed maximal subring with zero conductor which does not contain $K$ (see \cite[Lemma 4.6]{azarang6}). Hence $T/Q$ has an integrally closed maximal subring with zero conductor by Theorem \ref{exinicsc}. Thus $T$ has an integrally closed maximal subring with conductor $Q$.
\end{proof}

\begin{cor}
Let $K$ be an algebraically closed field which is not absolutely algebraic and $R$ be an affine ring over $K$. Then for each prime ideal $Q$ of $R$ with $ht(Q)\geq dim(R)-1$, there exists an integrally closed maximal subring $S$ of $R$ with $(S:R)=Q$.
\end{cor}

\begin{prop}
Let $K$ be a field which is not absolutely algebraic. The following are equivalent.
\begin{enumerate}
\item For each $n\geq 0$, the ring $K[X_1,\ldots, X_n]$ has an integrally closed maximal subring $S$ with zero conductor and $K\nsubseteq S$.
\item For each $n\geq 0$ and each prime ideal $Q$ of $R=K[X_1,\ldots, X_n]$, there exists an integrally closed maximal subring $S$ with $(S:R)=Q$ and $K\nsubseteq S$.
\item Each affine ring over $K$ has an integrally closed maximal subring $S$ with zero conductor and $K\nsubseteq S$.
\end{enumerate}
\end{prop}
\begin{proof}
It suffices to prove $(1)\Longrightarrow (2)$. Let $Q$ be a nonzero prime ideal of $R$, and $ht(Q)=m$. Hence $1\leq m\leq n$. If $m=n$, then $Q$ is a maximal ideal of $R$, therefore the field $R/Q$ contains a copy of $K$. Thus $R/Q$ is not absolutely algebraic field. Therefore by Corollary \ref{icmsnaf}, $R/Q$ has an integrally closed maximal subring. Hence $R$ has an integrally closed maximal subring $S$ which contains $Q$. Thus $Q\subseteq (S:R)$ and therefore $(S:R)=Q$, for $Q$ is a maximal ideal of $R$. Hence assume that $1\leq m\leq n-1$. Thus $d:=dim(\frac{R}{Q})=n-m\geq 1$. Therefore by Noether's Normalization Theorem (see \cite[Theorem A1 P.221 or Theorem 13.3]{eisnb}), there exist $Y_1,\ldots, Y_n$ in $R$ such that $R$ is integral over $T:=K[Y_1,\ldots,Y_n]$ and $P:=Q\cap T=(Y_{d+1},\ldots, Y_{n})$. Thus $B:=R/Q$ is integral over $A:=K[Y_1,\ldots, Y_d]$. Now by $(1)$, $A$ has an integrally closed maximal subring $S$ with zero conductor and $K\nsubseteq S$. Thus by Theorem \ref{exinicsc}, $B$ has an integrally closed maximal subring $V$ with zero conductor $K\nsubseteq V$. Hence $R$ has an integrally closed maximal subring with conductor $Q$, and we are done.
\end{proof}

Let $T$ be a ring, then we denote the set of all integrally closed maximal subrings of $T$ by $X^{i.c}(T)$ and also define $Spec(X^{i.c}(T)):=\{(S:T)\ |\ S\in X^{i.c}(T)\ \}$. Note that by $(3)$ of Theorem \ref{pret1}, $Spec(X^{i.c}(T))\subseteq Spec(T)$. Therefore $P\in Spec(X^{i.c}(R))$ if and only if $T$ has an integrally closed maximal subring $S$ with $(S:T)=P$. Now we have the following.

\begin{thm}
Let $R\subseteq T$ be an integral extension of rings. If $Max(T)\subseteq Spec(X^{i.c}(T))$, then $Max(R)\subseteq Spec(X^{i.c}(R))$.
\end{thm}
\begin{proof}
Let $P$ be a maximal ideal of $R$, then there exists a maximal ideal $Q$ of $T$ such that $Q\cap R=P$. Therefore $R/P\subseteq T/Q$ is an integral extension of fields. Now by our assumption there exists an integrally closed maximal subring $V$ of $T$ with $(V:T)=Q$. Hence $V/Q$ is an integrally closed maximal subring of $T/Q$. Therefore $R/P\nsubseteq V/Q$. Thus by Lemma \ref{cicmstf}, $V/Q\cap R/P$ is an integrally closed maximal subring of $R/P$. Hence $R$ has an integrally closed maximal subring $W$ which contains $P$ and therefore $(W:R)=P$.
\end{proof}

\begin{prop}
Let $R$ be an integrally closed maximal subring of an integral domain $T$ with $U(T)\nsubseteq R$. If $R$ is noetherian, then $(R:T)=0$. In particular, $R_M$ is a DVR, where $M$ is the crucial maximal ideal of the extension $R\subseteq T$.
\end{prop}
\begin{proof}
Assume that $t^{-1}\in U(T)\setminus R$. Therefore $t\in R$ and $R[\frac{1}{t}]=T$. Now if $P=(R:T)\neq 0$, then $tP=P$ for $t\in U(T)$. Therefore $t^{-1}P=P$. This immediately implies that $t^{-1}$ is integral over $R$ which is absurd. The final part is evident by the fact that $R$ is noetherian and $(5)$ of Theorem \ref{pret1}.
\end{proof}

\begin{rem}
Note that one can prove the first part of the previous proposition by Krull Intersection Theorem and the fact that $(R:T)=\bigcap_{n=1}^\infty Rt^n$.
\end{rem}

\begin{cor}
Let $R$ be a noetherian one dimensional integral domain. Then for each $0\neq a\in R\setminus U(R)$, the overring $R_a=R[\frac{1}{a}]$ has a maximal subring with zero conductor. In particular, if $X$ is a multiplicatively closed set of $R$ which is not contained in $U(R)$, then $R_X$ has a maximal subring with zero conductor.
\end{cor}
\begin{proof}
Let $T=R[\frac{1}{a}]$, then clearly $T$ has a maximal subring $S$ which contains $R$ and $\frac{1}{a}\notin S$. Therefore $S_a=T$. Now note that by Krull-Akizuki Theorem (\cite[Theorem 93]{kap}), $S$ is noetherian and therefore by Krull Intersection Theorem we conclude that $(S:T)=\bigcap_{n=1}^\infty Sa^n=0$. The final part is similar and follows
from Corollary \ref{cumcasg}.
\end{proof}

\begin{prop}\label{nnidicmszcdvr}
Let $R$ be a normal noetherian integral domain which is an integrally closed maximal subring of a ring $T$ with crucial maximal ideal $M$. Then $(R:T)=0$ and $R_M$ is a DVR.
\end{prop}
\begin{proof}
First note that by \cite[Theorem 10]{modica}, $T$ is an integral domain and therefore $T$ is a minimal overrring of $R$. Hence by \cite[Theorem 2.4 and Lemma 2.8]{aych}, there exists an ideal $A$ of $R$ such that $(R:T)=\bigcap_{n=1}^\infty A^n$. Thus by Krull intersection theorem we infer that $(R:T)=0$ and therefore by $(5)$ of Theorem \ref{pret1}, $R_M$ is a valuation domain. Since $R$ is noetherian, we conclude that $R_M$ is a DVR.
\end{proof}

\begin{prop}
Let $R$ be a completely integrally closed integral domain which is a conch maximal subring of a ring $T$. Then $(R:T)=0$.
\end{prop}
\begin{proof}
First note that there exists $a\in R$ such that $T=R[\frac{1}{a}]$, for $R$ is a conch maximal subring. Now one can easily see that $(R:T)=\bigcap_{n=1}^\infty Ra^n$. Thus by \cite[Corollary 13.4]{gilbok}, we infer that $(R:T)=0$ for $R$ is completely integrally closed.
\end{proof}

\begin{lem}\label{lmpmzcat}
Let $R$ be an integral domain with $0\neq (p)\in Spec(R)$. Then the following hold.
\begin{enumerate}
\item $R$ is a maximal subring of $R[\frac{1}{p}]$ with zero conductor if and only if $\bigcap_{n=1}^\infty (p^n)=0$ and $(p)\in Max(R)$.
\item if $R$ is atomic, then $(R:R[\frac{1}{p}])=0$. Moreover, $R$ is a maximal subring of $R[\frac{1}{p}]$ if and only if $(p)\in Max(R)$.
\end{enumerate}
\end{lem}
\begin{proof}
Assume that $T=R[\frac{1}{p}]$, then one can easily see that $(R:T)=\bigcap_{n=1}^\infty Rp^n$. Hence $(R:T)=0$ if and only if $\bigcap_{n=1}^\infty Rp^n=0$. For $(1)$, first suppose that $\bigcap_{n=1}^\infty Rp^n=0$ and $(p)\in Max(R)$, then we prove $R$ is a maximal subring of $T$. Let $A$ be a subring of $T$ which properly contains $R$. Let $x\in A\setminus R$, thus we may assume that $x=\frac{r}{p^n}$, where $r\in R$, $n\geq 1$ and $a\notin Rp$, for $\bigcap_{n=1}^\infty Rp^n=0$ and $x\in A\subseteq T$ but $x\notin R$. Hence
$\frac{r}{p}\in A$. Since $M=(p)$ is a maximal ideal of $R$, we conclude that $ap+br=1$, for some $a,b\in R$. Therefore $\frac{1}{p}=a+b\frac{r}{p}\in A$, and hence $T\subseteq A$. Thus $R$ is a maximal subring of $T$. Conversely, assume that $R$ is a maximal subring of $T$. Clearly $R$ conches $\frac{1}{p}$ in $T$. Therefore if $M$ is the crucial maximal ideal of the extension of $R\subseteq T$, then $M$ is the unique prime ideal of $R$ which contains $p$, by Corollary \ref{corconch}. Therefore $M=(p)$.  For $(2)$ note that by the proof of Corollary \ref{icmsipat}, $\bigcap_{n=1}^\infty Rp^n=0$ and therefore we are done by part $(1)$.
\end{proof}

\begin{prop}
Let $R\subseteq T$ be an integral extension of integral domain with $p\in R$ is prime element in $T$ and $pR=R\cap pT$. If $R$ is a maximal subring of $R[\frac{1}{p}]$ with zero conductor, then $T$ is a maximal subring of $T[\frac{1}{p}]$ with zero conductor.
\end{prop}
\begin{proof}
Since $R$ is a maximal subring of $R_1:=R[\frac{1}{p}]$ with zero conductor, we conclude that $0=(R:R_1)=\bigcap_{n=1}^\infty Rp^n$. Let $T_1=T[\frac{1}{p}]$, we claim that
$Q:=(T:T_1)=\bigcap_{n=1}^\infty Tp^n=0$. If $Q\neq 0$, then $Q\cap R\neq 0$ for $T$ is an integral domain which is integral over $R$. Let $0\neq x\in Q\cap R$, then for each $n\geq 1$, there exists $t_n\in T$ such that $x=p^n t_n$. Thus $\frac{x}{p^n}=t_n\in R[\frac{1}{p}]$ is integral over $R$. Since $R$ is integrally closed in $R_1$ we deduce that $t_n\in R$ and therefore $x\in\bigcap_{n=1}^\infty Rp^n=0$, which is absurd. Thus $Q=0$ and hence we are done by Lemma \ref{lmpmzcat}.
\end{proof}

\begin{rem}
Theorem \ref{icwtphtn} raises a natural question for the dual concept of maximal subrings, i.e., minimal ring extensions, as follows. Let $K$ be a field and $R=K[X_1,\ldots,X_n]$, where $X_1,\ldots,X_n$ are independent indeterminate over $K$. Assume that $P\in Spec(R)$ is arbitrary prime ideal. Now the natural question arises: Does there exist a minimal ring extension $T$ of $R$ with $(R:T)=P$? If $P\in Max(R)$ (for arbitrary ring $R$), then the answer to this question is positive by \cite[Corollary 2.5]{dobs}, and in this case note that by $(2)$ of Theorem \ref{pret1}, $T$ is integral over $R$. In fact by \cite[Corollary 2.5]{dobs}, for each maximal ideal $M$ of $R$, the idealization $T:=R(+)\frac{R}{M}$ is a minimal ring extension of $R$ with $(R:T)=M$. Now suppose $P$ is a prime ideal of $R=K[X_1,\ldots,X_n]$ which is not a maximal ideal of $R$. First note that by Lemma \ref{mscpnm}, if there exists a minimal ring extension $T$ of $R$ with $(R:T)=P$, then $R$ is integrally closed in $T$ and $ht(P)=n-1$, by $(5)$ of Theorem \ref{pret1}. Also by \cite[Theorem 10]{modica}, $T$ is an integral domain and therefore $T$ is an overring of $R$. Now we have two cases. If $n=1$, then $P=0$ and one can easily see that for each irreducible polynomial $p(X_1)\in R=K[X_1]$, the overring $T:=R[\frac{1}{p(X_1)}]$ is a minimal ring extension of $R$ with $(R:T)=0$. But if $n\geq 2$, then by \cite[Proposition 6.1]{aych}, $R=K[X_1,\ldots,X_n]$ has no minimal overring and therefore $R$ has no minimal ring extension with conductor $P$.
\end{rem}

In \cite{sato}, the authors proved that if $R$ is a one dimensional noetherian domain, then $R$ has a minimal overring. In particular, if $R$ is affine integral domain over an infinite field $K$ then $R$ has a minimal overring if and only if either $dim(R)=1$ or $dim(R)\geq 2$ and $R$ has a maximal ideal $M$ of depth 1. We conclude this article by the following corollary.

\begin{cor}
Let $R$ be an affine normal integral domain over a field $K$. Suppose $R$ is not a field. Then $R$ has a minimal overring if and only if $tr.deg(R/K)=1$.
\end{cor}
\begin{proof}
First note that $dim(R)=tr.deg(R/K)$. Hence if $tr.deg(R/K)=1$, then we are done by \cite[Theorem 3]{sato} (even if $R$ is not normal). Conversely, assume that $R$ has a minimal overring. Then by Proposition \ref{nnidicmszcdvr}, $R$ has a height one maximal ideal. Therefore $dim(R)=1$ and hence we are done.
\end{proof}


\begin{thebibliography}{mm}
\bibitem{andzaf1}
D.D. Anderson, David F. Anderson and Muhammad Zafrullah, Factorization in integral domains, J. Pure App. Algebra, {\bf 69} (1990), 1-19.

\bibitem{andchun}
D.D. Anderson and Sangmin Chun, Some remark on principal prime ideals, Bull. Aust. Math. Soc. {\bf 83} (2011) 130-137.



\bibitem{aych}
A. Ayache, Minimal overrings of an integrally closed domain, Comm. Algebra, {\bf 31} (12) (2003) 5693-5714.




\bibitem{azarang3}
A. Azarang, Submaximal integral domains, Taiwanese J. Math., {\bf 17} (4) (2013) 1395-1412.





\bibitem{azarang6}
A. Azarang, On fields with only finitely many maximal subrings, Houston J. Math {\bf 43} (2) (2017) 299-324.


\bibitem{azkrm2}
A. Azarang, O.A.S. Karamzadeh, On the existence of maximal subrings in commutative artinian rings, J. Algebra Appl. {\bf 9} (5) (2010) 771-778.



\bibitem{azkrm}
A. Azarang, O.A.S. Karamzadeh, Which fields have no maximal subrings?, Rend. Sem. Mat. Univ. Padova, {\bf 126} (2011) 213-228.


\bibitem{azkarm3}
A. Azarang, O.A.S. Karamzadeh, On Maximal Subrings of Commutative Rings, Algebra Colloquium, {\bf 19} (Spec 1) (2012) 1125-1138.

\bibitem{azkarm4}
A. Azarang, O.A.S. Karamzadeh, Most Commutative Rings Have Maximal Subrings, Algebra Colloquium, {\bf 19} (Spec 1) (2012) 1139-1154.


\bibitem{akn}
A. Azarang, O.A.S. Karamzadeh, A. Namazi, Hereditary properties between a ring and its maximal subrings, Ukrainian. Math. J. {\bf 65} (7) (2013) 981-994.


\bibitem{azomn}
A. Azarang, G. Oman, Commutative rings with infinitely many maximal subrings, J. Algebra Appl. {\bf 13} (7) (2014) ID:1450037.



\bibitem{cahen2}
P.J. Cahen, D.E. Dobbs and T.J. Lucas, Characterizing minimal ring extensions, Rocky Mountain Journal of Math {\bf 41} (4) (2011) 1081-1125.


\bibitem{adjex}
L.I. Dechene, Adjacent Extensions of Rings, Ph.D. Dissertation, University of California, Riverside, (1978).

\bibitem{dobs}
D.E. Dobbs, Every commutative ring has a minimal ring extension, Comm. Algebra, {\bf 34} (10) (2006) 3875-3881.


\bibitem{eisnb}
D. Eisenbud, Commutative Algebra-with a view toward Algebraic Geometry, Graduate Text in Mathematics, {\bf 150}, Springer-Verlag New York, 1995.


\bibitem{efrat}
Ido Efrat, Valuation, Ordering, and Milnor $K$-theory, Mathematical Surveys and Monographs, Volume 124, Amer. Math. Soc. (2006).



\bibitem{faith}
Carl Faith, The structure of valuation rings, J. Pure Appl. Algebra, {\bf 31} (1984) 7-27.


\bibitem{frd}
D. Ferrand, J.-P. Olivier, Homomorphismes minimaux d’anneaux, J. Algebra {\bf 16} (1970) 461-471.


\bibitem{gilbok}
R. Gilmer, Multiplicative Ideal Theorey, Marcel-Dekker 1972.


\bibitem{gilmer3}
R. Gilmer, Some finiteness conditions on the set of overring an integral domain. Proc. A.M.S. {\bf 131} (2003) 2337-2346.


\bibitem{gilmer}
R. Gilmer and W.J. Heinzer, Intersection of quotient rings of an integral domain. J. Math Kyoto Univ. {\bf 7} (2) (1967) 133-150.

\bibitem{grm}
A. Grams, Atomic domains and the ascending chain condition for principal ideals, Math. Proc. Cambridge Philos. Soc. {\bf 75} (1974) 321-329.


\bibitem{grifn}
M. Griffin, Prufer rings with few zero divisors, J. Reine Angew Math. 239-40 (1970) 55-67.

\bibitem{kap}
I. Kaplansky, Commutative Rings, rev. ed. University of Chicago Press, Chicago, 1974.


\bibitem{krul}
W. Krull, idealtheorie,Ergebnisse der Math. Wiss. und Ihrer Grenzgeb. 4 (Springer, Berlin, 1935).

\bibitem{mau}
S. Maubach and I. Stampfli, On Maximal Subalgebras, J. Algebra, {\bf 483} (2017) 1-36.

\bibitem{modica}
M.L. Modica, Maximal Subrings, Ph.D. Dissertation, University of Chicago, 1975.

\bibitem{abmin}
G. Picavet and M. Picavet-L'Hermitte. About minimal morphisms, in: Multiplicative ideal theory in commutative algebra, Springer-Verlag, New York, (2006) pp. 369-386.

\bibitem{sato}
J. Sato, T. Sugatani and K.-I. Yoshida, On minimal overring of a noetherian  domain, Comm. Algebra. {\bf 20} (6) (1992) 1735-1746.

\bibitem{invsub}
J. Szigeti and L.V. Wyk, Subrings which are closed with respect to taking the inverse, J. Algebra, {\bf 318} (2007) 1068-1076.

\end{thebibliography}

\end{document}